\documentclass{amsproc}

\usepackage{amscd,amssymb, amsmath, wasysym}
\usepackage{graphicx}
\usepackage{amsfonts}
\usepackage{mathrsfs}    % [00pt]changes size of font
\usepackage{amsmath}    % math enhancements latex2e (replaces amstex)
\usepackage{amsthm}     % proclaims theoremstyle/proof environments
\usepackage{amscd}      % commutative Diagram environment
\usepackage{amssymb}    % AMSFonts and symbols
\usepackage{eucal}      % Euler Cal/Script Fonts
\usepackage{latexsym}   % latex symbols (like \pounds)
\usepackage{graphicx}   % graphic package for postscript figures
\usepackage{verbatim}   % for comment environment
\usepackage[all]{xy}     % commutative diagram

\newtheorem{theorem}{Theorem}[section]

\newtheorem{proposition}[theorem]{Proposition}
\newtheorem{corollary}[theorem]{Corollary}

\theoremstyle{definition}
\newtheorem{definition}[theorem]{Definition}
\newtheorem{example}[theorem]{Example}

\newtheorem{conjecture}[theorem]{Conjecture}
\newtheorem{construction}[theorem]{Construction}

\theoremstyle{remark}
\newtheorem{remark}[theorem]{Remark}

\numberwithin{equation}{section}

\newtheoremstyle{claim}{9pt}{3pt}{}{\parindent}{\bf}{.}{1em}{}

\theoremstyle{claim}
\newtheorem{claim}[equation]{Claim}

\let \cedilla =\c
\newcommand{\nR}{\mathbb{R}}                     % real number
\newcommand{\nQ}{\mathbb{Q}}                     % complex number
\newcommand{\nP}{\mathbb{P}}                     % projective space

\newcommand{\nA}{\mathbb{A}}                     % affine space

\newcommand{\sG}{\mathscr{G}}
\newcommand{\sO}{\mathscr{O}}                    % structure sheaf

\newcommand{\sI}{\mathscr{I}}                    % ideal sheaf

\newcommand{\sP}{\mathscr{P}}

\newcommand{\mf}[1]{\mathfrak{#1}}

\DeclareMathOperator{\Bl}{Bl}                    % Blowup
\DeclareMathOperator{\bl}{Bl}                    % Blowup
\DeclareMathOperator{\codim}{codim}              % codim
\DeclareMathOperator{\glct}{glct}                % glct

\DeclareMathOperator{\exc}{exc}            % exc

\DeclareMathOperator{\htt}{ht}                   % ht

\DeclareMathOperator{\Jac}{Jac}                  % Jac

\DeclareMathOperator{\lct}{lct}                  % log canoical threshold
\DeclareMathOperator{\mld}{mld}                  % mld
\DeclareMathOperator{\mldmj}{mld_{MJ}}                % mld_MJ
\DeclareMathOperator{\ord}{ord}                  % ord

\DeclareMathOperator{\Spec}{Spec}                % Spec
\DeclareMathOperator{\spec}{Spec}                % Spec
\DeclareMathOperator{\Sing}{Sing}                % Sing
\DeclareMathOperator{\Exc}{Exc}                  % Exceptioinal locus
\DeclareMathOperator{\rank}{rank}                % rank

\newcommand*{\verylongrightarrow}{\ensuremath{\joinrel\relbar\joinrel\relbar\joinrel\relbar\joinrel\relbar\joinrel\rightarrow}}
\newcommand*{\superlongrightarrow}{\ensuremath{\joinrel\relbar\joinrel\relbar\joinrel\relbar\joinrel\relbar\joinrel\relbar\joinrel\relbar\joinrel\relbar\joinrel\rightarrow}}
\newcounter{rkcounter}             % set remark counter
\begin{document}

\title{A strongly geometric general residual intersection}

\author{Shihoko Ishii}
\address{Department of Mathematical Sciences, Tokyo Woman's Christian University,
	Tokyo, 167-8585, Tokyo}
%\curraddr{Department of Mathematics and Statistics,
%Case Western Reserve University, Cleveland, Ohio 43403}
\email{shihoko@lab.twcu.ac.jp}

\author{Wenbo Niu}
\address{Department of Mathematical Sciences, University of Arkansas, Fayetteville, AR 72701, USA}
\email{wenboniu@uark.edu}

%    General info
%\date{January 1, 1994 and, in revised form, June 22, 1994.}

\dedicatory{Dedicated to Professor Lawrence Ein on the occasion of his sixtieth birthday.}

\footnotetext{The first author  is partially supported by Grant-In-Aid (c) 1605089 of JSPS.}

\begin{abstract}
In this paper, we give formulas for the Grauert-Riemenschneider canonical sheaf and the log canonical threshold for a general residual intersection, and show that minimal log discrepancies are preserved under a general link. We also find evidences suggesting that MJ-singularities are preserved under a general residual intersection. 
\end{abstract}

\maketitle

%\section*{This is an unnumbered first-level section head}
%This is an example of an unnumbered first-level heading.
%
%\specialsection*{This is a Special Section Head}
%This is an example of a special section head%
%%%%%%%%%%%%%%%%%%%%%%%%%%%%%%%%%%%%%%%%%%%%%%%%%%%%%%%%%%%%%%%%%%%%%%%%%
%\footnote{Here is an example of a footnote. Notice that this footnote
%text is running on so that it can stand as an example of how a footnote
%with separate paragraphs should be written.
%\par
%And here is the beginning of the second paragraph.}%
%%%%%%%%%%%%%%%%%%%%%%%%%%%%%%%%%%%%%%%%%%%%%%%%%%%%%%%%%%%%%%%%%%%%%%%%%
\section{Introduction}

\noindent The purpose of this paper is to show how general equations can be chosen to produce a residual intersection for a variety and to investigate what properties and invariants of singularities can be preserved under this procedure. The concept of a residual intersection was introduced by Artin-Nagata \cite{Artin:ResidInterCM} in 1972. Its important case, namely linkage or liaison, was systematically studied by Peskine-Szpiro \cite{PeskineSzpiro:LiaisonI} to initiate the modern research of this area. The geometric idea behind the notion of a residual intersection is pretty natural: roughly speaking, any two varieties can be viewed as a residual intersection to each other in their union. When they have the same dimension, they are linked; otherwise, the one of smaller dimension is the residual intersection of another one.

More precisely, let $X$ be a closed subvariety   of a nonsingular variety $A$ of codimension $c$. By choosing $t$ equations from the defining ideal $I_X$ of $X$, we can define a closed subscheme $M\subset A$ containing $X$. Then the closure of $M\setminus X$ can be thought of as a $t$-residual intersection of $X$. Certainly, different choices of these $t$ equations result in different residual intersections. Among all possible residual intersections, it would be natural to expect that the one given by general equations would have the  most typical behavior. This is the idea that motivates our research in this paper.

In the work by Huneke and Ulrich (\cite{HU88}, \cite{Huneke:StrngCMredsiualInter}, \cite{HunekeUlrich:AlgLinkage}, \cite{HunekeUlrich:DivClass}, \cite{HunekeUlrich:SturLinkage}, etc), the rigorous  notions of generic linkage and residual intersection have been established and many fundamental properties have been proved. One of the central problems is to understand how algebraic or geometric properties of a given variety can be preserved under the procedure of generic linkage or residual intersection.  As for singularities, the behavior of rational singularities under residual intersection has been studied by Chardin and  Ulrich \cite{CU:Reg} in connection to questions about Castelnuovo-Mumford regularity. In the second author's work \cite{Niu:SingLink} and \cite{Niu:MJLinkage}, it has been shown that singularities such as log-canonical and MJ-singularities are preserved under generic linkage. However, it is not clear if similar results can be established for the more general situation of residual intersections. 

In this paper, we focus ourselves on studying singularities and their invariants under a general residual intersection.  To put the problem in perspective, we first give a construction of a general residual intersection, where our residual intersection is slightly different from the one defined by Huneke-Ulrich. Then based on techniques from resolutions of singularities, we discuss a couple of problems related to Grauert-Riemenschneider (GR) canonical sheaves, log canonical thresholds and minimal log discrepancies under a general residual intersection. Recall that $X\subset A$ is a closed subvariety. Let $Y$ be a general $t$-residual intersection (see Definition \ref{def:01}) with $t\geq c$. We prove in Theorem \ref{p:31} that the GR canonical sheaf of $Y$ is
$$\omega^{GR}_Y=\sI(A,I^t_X)\cdot \sO_Y\otimes \omega_A,$$
where $\sI(A,I^t_X)$ is the multiplier ideal associated to the pair $(A,I^t_X)$, and the log canonical threshold of $Y$ increases, i.e.,
$$\lct(A,Y)\geq \lct (A,X).$$ If we impose the extra condition that $X$ is locally a complete intersection with rational singularities, the result of Chardin-Ulrich \cite[Theorem 3.13]{CU:Reg} implies that $Y$ has rational singularities. In this case, the GR canonical sheaf $\omega^{RG}_Y$ is the same as the canonical sheaf $\omega_Y$ and therefore we obtain a formula for the canonical sheaf of a general residual intersection (Corollary \ref{p:33}). 

Turning to the local case, let $x\in X$ be a closed point. In Theorem \ref{local}, we show that general equations can be chosen to compute the minimal log discrepancies at $x$. In particular, in Corollary \ref{p:01}, we prove that if a complete intersection $M$ is defined by $c$ general equations $a_1,\cdots, a_c$ of $I_X$, then  
$$\mldmj(x;X, \mf{a}^{\bf{m}}) = \mldmj(x;M, \mf{a}^{\bf{m}}) = \mldmj(x;H, \mf{a}^{\bf{m}}),$$
where the hypersurface $H$ is defined by the product $a_1\cdots a_c$ and $\mf{a}^{\bf{m}}$ is a formal product of ideals. Hence the computation of minimal log discrepancies can be reduced to the complete intersection case or, even better, the hypersurface case. 

At the end of the paper, we provide evidence (Proposition \ref{p:34}) suggesting that MJ-singularities are preserved by general residual intersection (Conjecture \ref{conj:01}). The notion of MJ-singularities was introduced by Ishii, De Fernex-Docampo, and Ein-Ishii-Musta\cedilla{t}\v{a} (\cite{Ishii:MatherDis}, \cite{Roi:JDiscrepancy}, \cite{Ein:MultIdeaMatherDis}, and \cite{Ein:SingMJDiscrapency}). The advantage of this notion is that it can be established for arbitrary varieties without requiring normality and $\nQ$-Gorensteiness. We hope that the method in this paper can be useful for this conjecture.

\bigskip

\noindent {\em Acknowledgement:} We are grateful to Bernd Ulrich for valuable discussions. The second author would like to thank the University of Tokyo for the hospitality of visiting.

\section{General residual intersections}

\noindent Throughout the paper, we work over an algebraically closed field $k$ of characteristic zero. By a variety we mean an equidimensional reduced scheme of finite type over $k$.

\begin{definition}\label{def:01} Let $X$ be a closed subscheme of an nonsingular affine variety $A=\Spec R$ with ideal sheaf $I_X$. Let $I_M=(a_1,\cdots, a_t)\subseteq I_X$ such that it defines a closed subscheme $M$ of $A$. Let $Y$ be the closure of $M\setminus X$ (denoted by $Y=\overline{M\setminus X}$), where the closure of $M\setminus X$ is the closed subscheme defined by the coherent ideal sheaf maximal among ideal sheaves whose restriction on $A\setminus X$ coincide with the ideal sheaf of  $M\setminus X$. If $t=\codim Y\geq \codim X$, then $Y$ is called a {\em $t$-residual intersection}.
\end{definition} 

\begin{remark}
	Huneke and Ulrich defined a $t$-residual intersection in a different way (\cite{HU88}):
	the closed subscheme defined by $J=( I_M: I_X )$ is called a $t$-residual intersection if 
	$\htt(J)\geq t \geq \htt(I_X)$.
	Let $t$ be an integer such that $t\geq \codim(X, A)$, if $X$ satisfies $G_t$ and strongly Cohen-Macaulay,
	then our residual intersection is Cohen-Macaulay and coincides with the one defined by Huneke-Ulrich.
	This is proved by applying \cite[Theorem 3.1]{Huneke:StrngCMredsiualInter}.
\end{remark}

\begin{remark} Our definition of a residual intersection is in some sense ``geometric", so it would be nice to call it 
	a  geometric $t$-residual intersection.
	But the terminology ``geometric residual intersection" is already used for a different meaning in  \cite{HU88}.
	A suggestive alternative is to call it a ``strongly geometric $t$-residual 
	intersection" as is written in the title of this paper.
	But it is too long to be used frequently, so in this paper we call it just ``residual 
	intersection".

\end{remark}

We are interested in understanding how properties of a variety can be ``transfered" to  its $t$-residual intersection. Certainly, in Definition \ref{def:01}, a different choice of the generators of $I_M$ will result in a different $t$-residual intersection. It would be wonderful if the desired properties can be established for all $t$-residual intersections under every possible choice of $I_M$. But more realistically, we would like to investigate so called ``general" $t$-residual intersection with respect to a generating set of $I_X$. We make this point clear in the following construction. 

\begin{construction}\label{def:02}  Let $X$ be a closed subscheme of a nonsingular affine variety $A=\Spec R$ defined by an ideal $I_X$. Fix a generating set $\underline{f}=(f_1,\cdots, f_r)$ of $I_X$. Let $\sP$ be a property that we are interested in. For $t\geq \codim_AX$, by saying that a general $t$-residual intersection $Y$ with respect to $\underline{f}$ has property $\sP$ we mean the following: given the ideal
	$$I_M:=(a_1,\cdots,a_t)=(c_{i,j})\cdot(f_1,\cdots,f_r)^T,$$
	in which
	$$a_i:=c_{i,1}f_1+c_{i,2}f_2+\cdots+c_{i,r}f_r,\quad\quad\quad\mbox{for } 1\leq i\leq t,$$
	and $(c_{i,j})$ is a $t\times r$ matrix with elements in $k$, then there is an open set $U$ in $\nA^{t\times r}$ such that for any $(c_{i,j})\in U$, the corresponding $t$-residual intersection $Y=\overline{M\setminus X}$ has property $\sP$.  By saying that a general $t$-residual intersection $Y$ has property $\sP$ without mentioning $\underline{f}$, we mean that it is true for any choice of a generating set $\underline{f}$.
\end{construction}

The main framework in this paper is to use resolution of singularities to turn a variety into a divisor and then choose general equations by using Bertini's theorem to produce a $t$-residual intersection. The crucial point of this approach is that it automatically yields a resolution of singularities of the residual intersection. We show this in the following construction.  
\begin{construction}\label{def:03} Let $X$ be a closed subscheme of a nonsingular affine variety $A=\Spec R$ defined by an ideal $I_X$. Fix a generating set $\underline{f}=(f_1,\cdots, f_r)$ of $I_X$.  So we have a surjective morphism 
	\begin{equation*} 
		v:\bigoplus^r \sO_A\stackrel{f_1,\cdots, f_r}{\verylongrightarrow} I_X\longrightarrow 0.
	\end{equation*}
	Take a log resolution $\varphi: A_1\rightarrow A$ of $I_X$ such that $I_X\cdot\sO_{A_1}=\sO_{A_1}(-E)$ where $E$ is an effective divisor such that $\Exc(\varphi)\cup E$ is a divisor with simple normal crossing (snc) supports. Pulling back by $\varphi$, we obtain a surjective morphism
	$$\overline{v}:\bigoplus^r \sO_{A_1}\stackrel{\varphi^*f_1,\cdots, \varphi^*f_r}{\superlongrightarrow} \sO_{A_1}(-E)\longrightarrow 0.$$
	So the linear system $\varphi^*|f_1,\cdots, f_r|$ has $E$ as its base locus and therefore it can be decomposed as
	$$\varphi^*|f_1,\cdots, f_r|=W+E,$$
	in which $W$ is a base point free linear system on $A_1$. Take $t$ general elements $\alpha_1,\cdots, \alpha_t$ from $W$ and each $\alpha_i$ defines an nonsingular effective divisor $F_i$ on $A_1$ by Bertini's theorem. So we have
	\begin{itemize}
		\item [(1)] the sections $\alpha_1,\cdots, \alpha_t$ cut out a nonsingular closed subscheme $Y_1$ (could be reducible or empty) of $A_1$ of codimension $t$;
		\item [(2)] the closed subscheme $Y_1$ meets $E$ and $\exc(\varphi)$ transversally;
		\item [(3)] the sections $\alpha_1,\cdots, \alpha_t$ correspond to elements $a_1,\cdots, a_t$ in the vector space $\langle f_1,\cdots, f_r\rangle$.
	\end{itemize}
	Set $I_M=(a_1,\cdots, a_t)$ and call the subscheme $Y=\overline{M\setminus X}$ a {\em general $t$-residual intersection} (with respect to $\underline{f}$). Note also that 
	\begin{itemize}
		\item [(4)] the restriction $\varphi|_{Y_1}:Y_1\rightarrow Y$ is a resolution of singualrities of $Y$;
		\item [(5)]$I_M\cdot \sO_{A_1}=I_{Y_1}\cdot \sO_{A_1}(-E).$
	\end{itemize}
\end{construction}

\begin{remark} (1) In Construction \ref{def:03}, by Bertini's theorem, we can obtain an open set $U$ such that for any choice of the $t\times r$ matrix $(c_{i,j})$ the corresponding general $t$-residual intersection $Y$ satisfies the listed properties (1)-(5). Once a property $\sP$ is imposed, we need to further shrink this open set $U$ to get the desired one in Construction \ref{def:02}. This kind of argument should be clear from the context, so we do not always mention it explicitly. 
	
	(2)  Roughly speaking, our Construction \ref{def:02} can be viewed as a specialization of the generic residual intersection constructed in \cite{HU88}.
\end{remark}
\begin{proposition} Let $X$ be a closed subscheme of an affine nonsingular variety $A$. Let $Y$ be a general $t$-residual intersection of $X$. Then $Y$ is empty if and only if $\overline{I_M}=\overline{I_X}$, where $\overline{\mf{a}}$ denotes the integral closure of an ideal $\mf{a}$. 
\end{proposition}
\begin{proof} We use the notation in Construction \ref{def:03}. If $Y$ is empty, then $Y_1$ must be empty. This in turn implies that $I_M\cdot \sO_{A_1}=I_X\cdot \sO_{A_1}$. Therefore $\overline{I_M}=\overline{I_X}$ since $\overline{I_M}=\varphi_*(I_M\cdot \sO_{A_1})$ and $\overline{I_X}=\varphi_*(I_X\cdot \sO_{A_1})$.
	On the other hand, assume $\overline{I_M}=\overline{I_X}$. Denote by $A^+$ the normalization of the blowup of $A$ along $X$. We have $I_M\cdot\sO_{A^+}=I_X\cdot \sO_{A^+}$. This implies that on $A\setminus X$, $I_M=\sO_A$ so $Y$ is empty.
\end{proof}

\begin{proposition}
	\label{image}
	Let $X$ be a closed subscheme of an affine nonsingular variety $A$ defined by an ideal $I_X$ and let $\underline{f}=(f_1,\cdots, f_r)$ be a generating set of $I_X$. Let $Y$ be a general $t$-residual intersection with respect to $\underline{f}$ as in Construction \ref{def:03}. Consider the regular map 
	$$\psi:A\setminus X\longrightarrow \nA^r$$ 
	induced by the regular functions $f_1,\cdots, f_r$. Let $V$ be the closure of the image of $\psi$.
	\begin{itemize}
		\item [(1)] If $V$ is not a cone in $\nA^r$, then $Y$ is nonempty if and only if $$\dim k[f_1,\cdots, f_r] \geq t.$$
		\item [(2)] If $V$ is a cone in $\nA^r$,  then 		
		$Y$ is nonempty if and only if $$\dim k[f_1,\cdots, f_r] > t.$$
	\end{itemize}
\end{proposition}
\begin{proof} Let $\pi: \nA^r\setminus \{0\} \to \nP^{r-1}$ be the canonical projection and let 
	$W\subset \nP^{r-1}$ be the image of $V$ by $\pi$. 
	Then $\dim W=\dim V +1 $ if $V$ is a cone in $ \nA^r$, while  $\dim W=\dim V$ otherwise.
	A pull-back of $t$ general hyperplane cuts by $\psi^{-1}\circ\pi^{-1}$ gives 
	$(A\setminus X)\cap M$ for general $M$.
	Therefore, $Y$ is non-empty if and only if $\dim W\geq t$.
	This is equivalent to that $\dim V\geq t+1$ if $V$ is a cone in $\nA^{r}$ and
	$\dim V \geq t$ if $V$ is not a cone in $\nA^r$. 
	The proposition follows from $V=\spec k[f_1,\cdots, f_r]$.
\end{proof}

\begin{example} For every closed subscheme $X\subset \nA^N$, there is a generator system 
	$\underline{g}$ 
	of $I_X$ such that, for any $t$ with $\codim (X, \nA^N)\leq t \leq N-1$, a general $t$-residual intersection of $X$ is nonempty.
	Indeed, if a generator system $\underline{f}=\{f_1,\ldots,f_r\}$ is given, then we can add more generators of the form
	$x_1f_1, x_2f_1,\ldots, x_Nf_1$ to obtain a bigger generator system $\underline{g}$.
	Then,  $$\dim k[f_1,\ldots,f_r, x_1f_1, x_2f_1,\ldots, x_Nf_1]=N,$$
	where $x_1,x_2,\ldots, x_N$ are the coordinate functions of $\nA^N$.
	By the proposition above, for any $t$ with $\codim (X, \nA^N)\leq t \leq N-1$, a general $t$-residual intersection is not empty.

\end{example}

\begin{example} Let $X\subset \nA^4$ be the cone over  two skew lines in $\nP^3$.
	Note that it is  a two dimensional variety defined by the ideal $I_X=(xz, xw, yz, yw)$.
	First take the generator system $\underline{f}=\{xz, xw, yz, yw\}$.
	Then, we observe that the closure $V$ of the image of 
	$$\psi:\nA^4 \setminus X\longrightarrow \nA^4$$ 
	is a cone and 
	has dimension $3 $. Then, a general 3-residual intersection of $X$ with respect to 
	$\underline{f}$ is empty 
	by the Proposition \ref{image}.
	We can also calculate  this directly by solving the equations.
	
	Next, we take another generator system $$\underline{g}=\{xz, xw, yz, yw, x(xz), y(xz), z(xz),
	w(xz) \}$$ as in the previous example. 
	Then, $\dim k[\ \underline{g}\ ]=4$ and a general $3$-residual intersection of $X$ with respect to $\underline {g}$ is a curve.
	
\end{example}

\section{Invariants of singularities  under a general residual intersection}

\noindent In this section, we prove our main results about general residual intersections. We start with recalling some basic definitions. Let $X$ and $Z$ be subvarieties of a nonsingular variety $A$. Take a log resolution $f:A'\rightarrow A$ of $I_X\cdot I_Z$ such that $f^{-1}(X)=\sum^s_{i=1} a_iE_i$, $f^{-1}(Z)=\sum^s_{i=1}z_iE_i$ and the relative canonical divisor $K_{A'/A}=\sum^s_{i=1} k_iE_i$.  The {\em log canonical threshold} of $(A,X)$ is defined to be
$$\lct(A,X)=\min_i\left\{\frac{k_i+1}{a_i}\right\}.$$ 
Having fixed $\lambda\in \nR_{+}$, we also define the {\em generalized log canonical threshold} of $(A,X;\lambda Z)$ by
$$\glct(A,X;\lambda Z)=\min\left\{\frac{k_i+1+\lambda z_i}{a_i}\right\}.$$
For $c\in \nR_{+}$, the {\em multiplier ideal sheaf} $\sI(A,cX)$ associated to the pair $(A,cX)$ is defined by
$$\sI(A,cX)=f_*\sO_{A'}(K_{A'/A}-\lfloor c\sum a_iE_i\rfloor),$$
where $\lfloor c\sum a_iE_i\rfloor$ is the round down of the $\nR$-divisor $c\sum a_iE_i$.

Let $W$ be a proper closed subset of $A$ and let ${\mf{a}}^{\bf{m}}={\mf{a}}_1^{m_1}{\mf{a}}_2^{m_2}\cdots {\mf{a}}_n^{m_n}$ be a formal product of ideals ${\mf{a}}_i\subseteq \sO_A$ with ${\bf{m}}=(m_1,\cdots, m_n)\in \nR^n_{\geq 0}$. We define the {\em minimal log discrepancy} of $(A,{\mf{a}}^{\bf{m}})$ along $W$ as
$$\mld(W;A,{\mf{a}}^{\bf{m}})=\inf_{C_A(E)\subseteq W}\{{a}(E;A,{\mf{a}}^{\bf{m}})+1\ |\ E \mbox{ a prime divisor over $A$} \}$$
where ${a}(E;A,{\mf{a}}^{\bf{m}})$ is the discrepancy of $E$ and $C_A(E)$ is the center of $E$ in $A$. We use the convention that if $\dim A=1$ and $\mld(W;A,{\mf{a}}^{\bf{m}})$ is negative, then we set it as $-\infty$. For more details on the invariants we defined above, we refer to, for instance, the work \cite{Ein:JetSch}.

We also use the notation of Mather-Jacobian singularities (MJ-singularities for short) which was introduced and studied in \cite{Ishii:MatherDis}, \cite{Roi:JDiscrepancy}, \cite{Ein:MultIdeaMatherDis} and \cite{Ein:SingMJDiscrapency}. Recall that $X$ is a variety of dimension $n$ and let $f:X'\longrightarrow X$ be a log resolution of the Jacobian ideal $\Jac_X$ of $X$. Then the image of the canonical homomorphism
$$f^*(\wedge^n\Omega^1_X)\longrightarrow \wedge^n\Omega^1_{X'}$$
is an invertible sheaf of the form $\Jac_f\cdot\wedge^n\Omega^1_X$, where $\Jac_f$ is the relative Jacobian ideal of $f$. The ideal $\Jac_f$ is invertible and defines an effective divisor $\widehat{K}_{X'/X}$ which is called the {\em Mather discrepancy divisor} (see also \cite[Remark 2.3]{Ein:MultIdeaMatherDis}).  For an ideal $\mf{a}\subseteq \sO_X$ and $t\in \nQ_{\geq 0}$ and for a prime divisor $E$ over $X$, consider a log resolution $\varphi:X'\longrightarrow X$ of $\Jac_X\cdot \mf{a}$ such that $E$ appears in $X'$ and $\mf{a}\cdot\sO_{X'}=\sO_{X'}(-Z)$ and $\Jac_X\cdot\sO_{X'}=\sO_{X'}(-J_{X'/X})$ where $Z$ and $J_{X'/X}$ are effective divisors on $X'$. We define the {\em Mather-Jacobian-discrepancy} ({\em MJ-discrepancy} for short) of $E$ to be
$$a_{\text{MJ}}(E;X,\mf{a}^t)=\ord_E(\widehat{K}_{X'/X}-J_{X'/X}-tZ).$$
The number $a_{\text{MJ}}(E;X,\mf{a}^t)+1$ is called the {\em Mather-Jacobian-log discrepancy} ({\em MJ-log discrepancy} for short). It is independent on the choice of the log resolution $\varphi$. Let $W$ be a proper closed subset of $X$ and let $\eta$ be a point of $X$ such that its closure $\overline{\{\eta\}}$ is a proper closed subset of $X$. We define the {\em minimal MJ-log discrepancy} of $(X,\mf{a}^t)$ along $W$ as
$$\mldmj(W;X,\mf{a}^t)=\inf_{C_X(E)\subseteq W}\{\ a_{\text{MJ}}(E;X,\mf{a}^t)+1\ |\ E \mbox{ a prime divisor over $X$} \}$$
and the {\em minimal MJ-log discrepancy} of $(X,\mf{a}^t)$ at $\eta$ as
$$\mldmj(\eta;X,\mf{a}^t)=\inf_{C_X(E)=\overline{\{\eta\}}}\{\ a_{\text{MJ}}(E;X,\mf{a}^t)+1\ |\ E \mbox{ a prime divisor over $X$} \}.$$
We say that $X$ is {\em MJ-canonical} (resp. {\em MJ-log canonical}) if for every exceptional prime divisor $E$ over $X$, the MJ-discrepancy $a_{\text{MJ}}(E;X,\sO_X)\geq 0$ (resp. $\geq -1$) holds. When $X$ is nonsingular, the notion of MJ-singularities is the same as the usual (log) canonical singularities. We shall use the following version of the Inversion of Adjunction. It plays a critical role in transferring singularity information from a variety to its ambient space.  
\begin{theorem}[Inversion of Adjunction, {\cite{Ishii:MatherDis}\cite{Roi:JDiscrepancy}}]Let $X$ be a codimension $c$ subvariety of a nonsingular variety $A$ defined by the ideal $I_X$.
	\begin{itemize}
		\item [(1)] Let $W\subset X$ be a proper closed subset of $X$. Then
		$$\mldmj(W;X,\sO_X)=\mld(W;A,I^c_X).$$
		\item [(2)] Let $\eta\in X$ be a point such that its closure $\overline{\{\eta\}}$ is a proper closed subset of $X$. Then
		$$\mldmj(\eta;X,\sO_X)=\mld(\eta ;A,I^c_X).$$
	\end{itemize}
\end{theorem}

Our first main result is to show that minimal log discrepancies can be computed by using a complete intersection or a hypersurface. This result sheds light on how to apply linkage theory to the study of singularities. 
\begin{theorem} \label{local}
	Let $x$ be a closed point of a nonsingular affine variety $A$. Let $I\subseteq \sO_A$ be an ideal generated by $\{f_1,\ldots,f_r\}$ and let $\tilde{\mf{a}}^{\bf{m}}=\tilde{\mf{a}}_1^{m_1}\tilde{\mf{a}}_2^{m_2}\cdots \tilde{\mf{a}}_n^{m_n}$ be a formal product of ideals $\tilde{\mf{a}}_i\subseteq \sO_A$ with ${\bf{m}}=(m_1,\cdots, m_n)\in \nR^n_{\geq 0}$. Then for $t\geq 1$ general elements $a_1,\cdots, a_t$ in the vector space $\langle f_1,\cdots, f_r \rangle$, the ideals $I_M=(a_1,\cdots,a_t)$ and $I_H=(a_1a_2\cdots a_t)$ satisfy the following properties.
	
	There exist a log resolution $\varphi:A_1\rightarrow A$ of $\tilde{\mf{a}}^{\bf{m}}\cdot I\cdot I_H\cdot \mf{m}_x$ and a log resolution $\psi:A_2\rightarrow A$ of $\tilde{\mf{a}}^{\bf{m}}\cdot I\cdot I_M\cdot I_H\cdot \mf{m}_x$ such that $\psi$ factors through $\varphi$, i.e., $\psi=\varphi\circ\mu$, and $\mu_*$ induces a one-to-one correspondence
	$$\{\text{prime divisors in $A_2$ with center $x$}\}\stackrel{\mu_*}{\longrightarrow} \{ \text{prime divisors in $A_1$ with center $x$}\}.$$
	Furthermore, for a prime divisor $F\subset A_1$ with center $x$ and a real number $0\leq \lambda\leq t$, 
	$$a(F;A,I^{\lambda}\cdot \tilde{\mf{a}}^{\bf{m}})=a(F;A, I^{\lambda}_M\cdot\tilde{\mf{a}}^{\bf{m}})=a(F; A,I_H^{\lambda/t}\cdot \tilde{\mf{a}}^{\bf{m}}), \text{ and}$$
	$$\mld(x;A,I^{\lambda}\cdot \tilde{\mf{a}}^{\bf{m}})=\mld(x;A, I^{\lambda}_M\cdot\tilde{\mf{a}}^{\bf{m}})=\mld(x; A,I_H^{\lambda/t}\cdot \tilde{\mf{a}}^{\bf{m}}).$$
\end{theorem}

\begin{proof}For simplicity, we prove the case $n=1$ and $\tilde{\mf{a}}^{\bf{m}}=\tilde{\mf{a}}^m$ .  Take a log resolution $\varphi: A_1\longrightarrow A$ of $\tilde{\mf{a}}\cdot I\cdot \mf{m}_x$ such that $I\cdot\sO_{A_1}=\sO_{A_1}(-E)$ and $\tilde{\mf{a}}\cdot \sO_{A_1}=\sO_{A_1}(-Z)$ where $E$ and $Z$ are effective divisors and $\Exc(\varphi)\cup E\cup Z$ has  a snc support. As in Construction \ref{def:03}, we have a decomposition 
	$$\varphi^*| f_1,\cdots, f_r |=W+E$$
	where $W$ is free linear system. Choose $t$ general elements $\alpha_1,\cdots, \alpha_t$ in $W$ and denote by $F_i$ the zero locus of $\alpha_i$. By Bertini's theorem, $F_i$ can be chosen such that $F_1\cup\cdots \cup F_t\cup \Exc(\varphi)\cup E\cup Z$ has a snc support and therefore the intersection $Y_1=F_1\cap\cdots \cap F_t$ is a nonsingular subscheme (not necessarily irreducible) in $A_1$ of codimension $t$. Furthermore, the sections $\alpha_1,\cdots, \alpha_t$ correspond to the general elements $a_1,\cdots, a_t$ as desired in the vector space $\langle f_1,\cdots, f_r \rangle$ such that 
	$$I_M\cdot \sO_{A_1}=I_{Y_1}\cdot \sO_{A_1}(-E).$$
	Now blowup $A_1$ along $Y_1$ to obtain $\mu: A_2=\bl_{Y_1}A_1\rightarrow A_1$ with an exceptional divisor $G$. Set $\psi=\varphi\circ\mu: A_2\rightarrow A$. Notice that
	$$K_{A_2/A_1}=(t-1)G$$
	Hence we obtain 
	$$K_{A_2/A}-\lambda I\cdot \sO_{A_2}-m\tilde{\mf{a}}\cdot \sO_{A_2}=\mu^*K_{A_1/A}+(t-1)G-\lambda\mu^*E-m\mu^*Z, \text{ and}$$
	$$K_{A_2/A}-\lambda I_M\cdot \sO_{A_2}-m\tilde{\mf{a}}\cdot \sO_{A_2}=\mu^*K_{A_1/A}-(t-1-\lambda)G-\lambda\mu^*E-m\mu^*Z.$$
	
	We shall use \cite[Proposition 7.2]{Ein:JetSch} to compute the $\mld$ at $x$. First notice that, if a prime divisor with center $x$ appears in $A_1$ then it must appear in $A_2$ and vice versa, which gives the one-to-one correspondence in the proposition. Let $F\subset A_1$ be such a prime divisor with center $x$. Then clearly $F$ is not contained in $Y_1$ and therefore  $F$ is not contained in $G$. Hence for any real number $0\leq \lambda\leq t$, we have $a(F;A,I^{\lambda}\cdot \tilde{\mf{a}}^m)=a(F;A, I^{\lambda}_M\cdot\tilde{\mf{a}}^m)$. 
	
	Next, we consider those prime divisors whose center contains $\{x\}$ as a proper subset. Let $F$ be such a prime divisor. We have two possibilities. First, $F$ is contained in $G$. In this case, $F$ cannot be in the support of $\mu^*K_{A_1/A}\cup\mu^*E\cup \mu^*Z$. So we have
	$$a(F;A,I^{\lambda}\cdot \tilde{\mf{a}}^m)=\ord_FK_{A_2/A}-\lambda\ord_FI-m\ord_F\tilde{\mf{a}}+1=(t-1)+1\geq 0, \text{ and}$$
	$$a(F;A, I^{\lambda}_M\cdot\tilde{\mf{a}}^m)=\ord_FK_{A_2/A}-\lambda\ord_FI_M-m\ord_F\tilde{\mf{a}}+1=(t-1-\lambda)+1\geq 0.$$
	The second possibility is that $F$ appears in $A_2$ but is not contained in $G$. In this case, $F$ must also appear in $A_1$. This implies $\ord_FI=\ord_FI_M$ and therefore  $a(F;A,I^{\lambda}\cdot \tilde{\mf{a}}^m)=a(F;A, I^{\lambda}_M\cdot\tilde{\mf{a}}^m)$ (could be negative). 
	
	All of the above implies that $$\mld(x;A,I^{\lambda}\cdot \tilde{\mf{a}}^m)=\mld(x;A, I^{\lambda}_M\cdot\tilde{\mf{a}}^m),$$
	as desired.
	
	For the case of $I_H=(a_1\cdots a_t)$, we notice that $\varphi$ is a log resolution of $I\cdot I_H\cdot \mf{m}_x$ such that 
	$$I_H\cdot \sO_{A_1}=F_1+\cdots +F_t+tE.$$
	Hence 
	$$K_{A_1/A}-\frac{\lambda}{t}\cdot I_H\cdot \sO_{A_1}-m\tilde{\mf{a}}\cdot\sO_{A_1}=K_{A_1/A}-\frac{\lambda}{t}F_1-\cdots-\frac{\lambda}{t}F_t-\lambda E-mZ.$$
	Now for a prime divisor $F$ contained in some $F_i$, the center $C_A(F)$ cannot be $x$ and hence 
	$$\ord_FK_{A_1/A}-\frac{\lambda}{t}\ord_F\cdot I_H-m\ord_F\tilde{\mf{a}}+1=-\frac{\lambda}{t}+1\geq  0.$$
	This proves the last equality. 
\end{proof}

As a quick corollary, we see that at a fixed point, minimal log discrepancies are preserved under a general link. For a related result without fixing a point, we refer to \cite{Niu:MJLinkage}.
\begin{corollary}\label{p:01}If $I=I_X$ is the ideal of a closed subvariety   $X\subseteq A$ of codimension $c$ and $t=\lambda=c$, then the closed subscheme $M$ defined by $I_M$ is a complete intersection containing $X$ and 
	$$\mld(x;A,I_X^c\cdot \tilde{\mf{a}}^{\bf{m}})=\mld(x;A, I^c_M\cdot\tilde{\mf{a}}^{\bf{m}})=\mld(x; A,I_H\cdot \tilde{\mf{a}}^{\bf{m}}),$$	
	which yields
	$$\mldmj(x;X, \mf{a}^{\bf{m}}) = \mldmj(x;M, \mf{a}^{\bf{m}}) = \mldmj(x;H, \mf{a}^{\bf{m}}),$$
	where $\mf{a}$ denotes the restrictions of $\tilde{\mf{a}}$ on $X$, $M$ and $H$. Also the prime divisors which compute the minimal log discrepancies are common. Moreover, let $Y=\overline{M\setminus X}$, then 
	$$\mldmj(x;Y,{\mf{a}}^{\bf{m}})= \mld(x;A,I_Y^c\cdot \tilde{\mf{a}}^{\bf{m}})\geq \mld(x;A,I_X^c\cdot \tilde{\mf{a}}^{\bf{m}})=\mldmj(x;X, {\mf{a}}^{\bf{m}}).$$	
\end{corollary}

\begin{remark} In \cite[Theorem 5.6, (ii)]{Ein:SingMJDiscrapency}, it is proved that a two dimensional non-complete
	intersection MJ-log canonical variety $X$ has a complete intersection model $M$, i.e., $M$ is a
	complete intersection containing $X$ and
	$$\mldmj(x;X) = \mldmj(x;M) = 0.$$
	Therefore, Corollary \ref{p:01} is a generalization of \cite[Theorem 5.6, (ii)]{Ein:SingMJDiscrapency}.
\end{remark}

In the second main result, we establish a formula for the  Grauert-Riemenschneider canonical sheaf of a general residual intersection. One of the central topics in the study of residual intersection is to obtain a formula for  canonical sheaves, which is usually difficult. However, the Grauert-Riemenschneider canonical sheaf is always contained in the canonical sheaf and they are the same under certain conditions. We also give a description of log canonical thresholds as well as its generalized form under general residual intersection, which says that after doing residual intersections, singularities improve.

\begin{theorem}\label{p:31} Let $X$ be a closed subvariety   of an affine nonsingular variety $A$ and let $Y$ be a general $t$-residual intersection of $X$. If $Y$ is nonempty, then one has
	\begin{itemize}		
		\item [(1)] $\omega^{GR}_Y=\sI(A,I_X^t)\cdot \sO_Y\otimes \omega_A$.
		\item [(2)] $\lct(A,Y)\geq \lct(A,M)=\lct(A,X)$.
		\item [(3)] $\glct(A,Y;(t-c)X)\geq \glct(A,M;(t-c)X)=\lct(A,X)+(t-c)$.
	\end{itemize}
\end{theorem}
\begin{proof} We use the notation in Construction \ref{def:03}.	Now the sections $\alpha_i$ give a surjective morphism
	$$\bigoplus^t \sO_{A_1}(E)\stackrel{\alpha_1,\cdots,\alpha_t}{\verylongrightarrow} I_{Y_1}\longrightarrow 0.$$
	Since $Y_1$ is a complete intersection in $A_1$, restricting the above map onto $Y_1$ we obtain the conormal bundle of $Y_1$ inside $A_1$ as
	$$N^*_{Y_1/A_1}=\bigoplus^t \sO_{Y_1}(E).$$
	Taking the determinant of $N^*_{Y_1/A_1}$ and tensoring with $\omega_{A_1}$ yields  the dualizing sheaf of $Y_1$ as
	$$\omega_{Y_1}=\omega_{A_1}\otimes \sO_{Y_1}(-tE).$$
	We also have a short exact sequence 
	$$0\longrightarrow I_{Y_1}\longrightarrow \sO_{A_1}\longrightarrow \sO_{Y_1}\longrightarrow 0$$
	which induces a short exact sequence
	\begin{equation}\label{eq:01}
		0\longrightarrow I_{Y_1}\otimes \sO_{A_1}(K_{A_1/A}-tE)\longrightarrow \sO_{A_1}(K_{A_1/A}-tE)\longrightarrow \omega_{Y_1}\otimes \varphi^*\omega^{-1}_A\longrightarrow 0.
	\end{equation}

	We further blowup $A_1$ along $Y_1$ as $\mu: \tilde{A}=\Bl_{Y_1}A_1\rightarrow A$ with an exceptional divisor $T$ on $\tilde{A}$ such that $I_{Y_1}\cdot \sO_{\tilde{A}}=\sO_{\tilde{A}}(-T)$. Write $\psi=\varphi\circ\mu: \tilde{A}\rightarrow A$ and set $$\sG=\sO_{\tilde{A}}(-T)\otimes \mu^*( \sO_{A_1}(K_{A_1/A}-tE)).$$
	
	\begin{claim} One has $R^1\varphi_*\sG=0$.
	\end{claim}
	{\em Proof of Claim}: Note that $\mu_*\sG=I_{Y_1}\otimes  \sO_{A_1}(K_{A_1/A}-tE)$ and for $i>0$, $R^i\mu_*\sG=R^i\mu_*\sO_{\tilde{A}}(-T)\otimes \sO_{A_1}(K_{A_1/A}-tE)=0$. On the other hand, notice that $K_{\tilde{A}/A_1}=(t-1)T$ and $K_{\tilde{A}/A}=K_{\tilde{A}/A_1}+\mu^*K_{A_1/A}$ so we have
	$$\sG=\sO_{\tilde{A}}(-T)\otimes \mu^*( \sO_{A_1}(K_{A_1/A}-tE))=\sO_{\tilde{A}}(K_{\tilde{A}/A}-t(T+\mu^*E)).$$
	But by construction, $I_M\cdot\sO_{A_1}=I_{Y_1}\cdot\sO_{A_1}(-E)$ and therefore $I_{M}\cdot \sO_{\tilde{A}}=\sO_{\tilde{A}}(-T-\mu^*E)$. Hence we obtain that $\sG=\sO_{\tilde{A}}(K_{\tilde{A}/A}-t(I_{M}\cdot \sO_{\tilde{A}}))$. By Local Vanishing Theorem \cite[9.4.1]{Lazarsfeld:PosAG2}, we then obtain that $R^i\psi_*(\sG)=0$, for $i>0$. Now we use the Leray spectral sequence
	$$E^{p,q}_2=R^p\varphi_*R^q\mu_*\sG\Rightarrow R^{p+q}(\varphi\circ\mu)_*(\sG)$$
	to get that  $R^1\varphi_*\sG=0$, completing the proof of claim.
	
	\bigskip

	Now we prove (1). For this, we push down the short exact sequence (\ref{eq:01}) and use the claim above. So we obtain an short exact sequence  
	$$0\longrightarrow \varphi_*(I_{Y_1}\otimes \sO_{A_1}(K_{A_1/A}-tE))\longrightarrow \sI(A,I^t_X)\longrightarrow \omega^{GR}_Y\otimes \omega^{-1}_A\longrightarrow 0.$$
	This implies that $\sI(A,I^t_X)\cdot \sO_Y=\omega^{GR}_Y\otimes \omega^{-1}_A$, which gives rise to the desired formula by tensoring with $\omega_A$. 
	
	For (2), recall that $K_{\tilde{A}/A}=(t-1)T+\mu^*K_{A_1/A}$, $I_M\cdot\sO_{\tilde{A}}=\sO_{\tilde{A}}(-T-\mu^*E), \text{ and }I_X\cdot\sO_{\tilde{A}}=\sO_{\tilde{A}}(-\mu^*E).$
	In order to compute $\lct(A,M)$, we consider prime divisors on $\tilde{A}$. For the divisor $T$, we have
	$$\frac{k_T+1}{\ord_TI_M}=\frac{t-1+1}{1}=t.$$
	For any prime divisor $F\neq T$ on $\tilde{A}$, we have $\ord_FI_M=\ord_FI_X$ and hence
	$$\frac{k_F+1}{\ord_FI_M}=\frac{k_F+1}{\ord_FI_X}.$$
	Thus
	$$\lct(A,M)=\min_{F\neq T}\{t, \frac{k_F+1}{\ord_FI_M}\}=\min_{F\neq T}\{t, \frac{k_F+1}{\ord_FI_X}\}=\min\{ t, \lct(A,X)\}.$$
	But since $X$ is generically smooth, we have $\lct(A,X)\leq c\leq t$. So we obtain that $\lct(A,M)= \lct(A,X)$. Finally since $I_M\subseteq I_Y$, we then have $\lct(A,Y)\geq \lct(A,M)=\lct(A,X)$.
	
	For (3), we continue to work on $\tilde{A}$. For the divisor $T$, we have
	$$ \frac{k_T+\ord_TI_X^{t-c}+1}{\ord_TI_M}=\frac{t-1+1}{1}=t.$$
	For any prime divisor $F\neq T$, we have $\ord_FI_M=\ord_FI_X$ and therefore
	$$ \frac{k_F+\ord_FI_X^{t-c}+1}{\ord_FI_M}=\frac{k_F+(t-c)\ord_FI_X+1}{\ord_FI_X}=\frac{k_F+1}{\ord_FI_X}+(t-c).$$
	Hence we see
	$$\glct(A,M;(t-c)I_X)=\min_{F\neq T}\{t,\frac{k_F+1}{\ord_FI_X}+(t-c) \}=\min\{t, \lct(A,X)+(t-c)\}.$$
	Since $X$ is a variety of codimension $c$, we have $\lct(A,X)\leq c$ and therefore $\lct(A,X)+(t-c)\leq t$. Hence $\glct(A,M;(t-c)I_X)=\lct(A,X)+(t-c)$, from which the result follows.
\end{proof}

\begin{corollary}\label{p:33} Let $X$ be a closed subvariety   in an affine nonsingular variety $A$. Assume that $X$ is locally a complete intersection with rational singularities. Then a general $t$-residual intersection $Y$ of $X$ with $t\geq \codim_AX$ has rational singularities and 
	$$\omega_Y=\sI(A,I^{t})\cdot \sO_Y\otimes \omega_A.$$
\end{corollary}
\begin{proof} Using \cite[Theorem 3.13]{CU:Reg}, $Y$ has rational singularities and therefore $\omega_Y^{GR}=\omega_Y$. Now the result follows from Theorem \ref{p:31}.
\end{proof}

The following proposition describes the singularities of a general $t$-residual intersection of a nonsingular variety. 
The statement (1) is proved in  \cite[Theorem 3.13]{CU:Reg} and (2) is proved 
in \cite[Corollary 8.1.4.]{FOV99}.
Here we prove them in different ways; along the line of the construction of general
residual intersections.

\begin{proposition} Let $X$ be a nonsingular closed subvariety   of a nonsingular affine  variety $A$ and let $Y$ be a general $t$-residual intersection of $X$. If $Y$ is nonempty, then 
	\begin{itemize}
		\item [(1)] $Y$ has rational singularities.
		\item [(2)] $\codim_Y\Sing(Y)\geq t-c+4$.
		\item [(3)] $\omega_Y=I_X^{t-c+1}\cdot\sO_Y\otimes\omega_A$.
	\end{itemize}
\end{proposition}

\begin{proof} We use Construction \ref{def:03}. Since $X$ is nonsingular, we can assume $X$ is irreducible. We can take $A_1$ as the blowup of $A$ along $X$ and therefore the divisor $E$ is irreducible.  Since $Y_1$ is a complete intersection in $A_1$, we have a Koszul resolution of $\sO_{Y_1}$ as follows	$$0\longrightarrow \sO_{A_1}(tE)\longrightarrow\cdots\longrightarrow \bigoplus^t\sO_{A_1}(E)\longrightarrow \sO_{A_1}\longrightarrow \sO_{Y_1}\longrightarrow 0.$$
	Note that $R^i\varphi_*\sO_{A_1}(jE)=0$ for $i>0$, $0\leq j\leq c-1$ and  $R^i\varphi_*\sO_{A_1}(jE)=0$ for $i\geq c$ and $j\geq 0$. Hence pushing down the complex above, we deduce that $R^i\varphi_*\sO_{Y_1}=0$ for $i>0$ and the map $\sO_{A}\rightarrow \varphi_*\sO_{Y_1}$ is surjective as well. But we have a natural subjection $\sO_A\rightarrow \sO_Y$ through which the map $\sO_{A}\rightarrow \varphi_*\sO_{Y_1}$ factors, i.e., $\sO_A\rightarrow \sO_Y\rightarrow \varphi_*\sO_{Y_1}$. Hence we conclude that $\varphi_*\sO_{Y_1}=\sO_Y$ and therefore $Y$ has rational singularities.
	
	It is clear that $\Sing(Y)\subseteq X\cap Y$. For any $x\in X$, denote by $E_x$ the fiber of the map $\varphi|_E:E\rightarrow X$. We have the following three cases for the intersection $Y_1\cap E_x$.
	\begin{itemize}
		\item [(a)] $Y_1\cap E_x=\emptyset$, which is equivalent to  $x\notin X\cap Y$. 
		\item [(b)] $Y_1\cap E_x$ is a closed point, which implies that $x\in X\cap Y$ and $Y$ is nonsingular at $x$.
		\item [(c)]	$Y_1\cap E_x$ has dimension $\geq 1$.
	\end{itemize}
	Since $E=\nP(I_X/I^2_X)$, the $t$ general sections give $t$ sections in $H^0(\sO_E(1))$. Consider the map
	$$\psi:\bigoplus^t\sO_X\longrightarrow I_X/I^2_X$$
	induced by the $t$ sections. We see that the locus of points $x$ satisfying case (c) is inside 
	$$\sigma_{c-2}=\{x\in X|\rank \psi\otimes k(x)\leq c-2\}.$$
	Hence $\Sing(Y)\subseteq \sigma_{c-2}$. But the codimension of $\sigma_{c-2}$ in $X$ is $(t-(c-2))(c-(c-2))$. So we obtain $\codim_Y\Sing(Y)\geq t-c+4$.
	
	Now for (3), since $X$ is nonsingular, the multiplier ideal $\sI(A,I^t_X)=I_X^{t-c+1}$. It follows from Proposition \ref{p:31} and $Y$ having rational singularities that $\omega_Y=\omega^{GR}_Y=I_X^{t-c+1}\cdot\sO_Y\otimes\omega_A$.
	
\end{proof}

At the end of this section, we discuss which singularities might  be preserved under a general residual intersection. Chardin-Ulrich's result (\cite{CU:Reg}) says that a general residual intersection of a local complete intersection with rational singularities also has rational singularities. A complete answer on rational singularities under a general link has been given in \cite{Niu:SingLink}. However, the situation is not clear for a general residual intersection. We point out that a local complete intersection with rational singularities has MJ-canonical singularities. In the generic linkage case, it has been proved that MJ-singularities are preserved. So it is reasonable to expect that a similar result can be established for a general residual intersection. Along this line, we provide the following evidence. 

\begin{proposition}\label{p:34} Let $X$ be a closed subvariety   of dimension $\leq 3$  of a nonsingular affine  variety $A$. Assume that $X$ is locally a complete intersection with MJ-log canonical (resp, MJ-canonical) singularities. Then a general  $t$-residual intersection $Y$ of $X$ with $t\geq \codim_AX$ is also locally a complete intersection with MJ-log canonical (resp. MJ-canonical) singularities. 
\end{proposition}
\begin{proof}We keep using the notations in Construction \ref{def:03}. If a general $t$-residual intersection is empty, then there is nothing to prove. So we assume in the sequel that a general $t$-residual intersection is not empty. Let $Y_c$ be a general link of $X$. Since $\dim X\leq 3$, by \cite[Proposition 2.9]{HunekeUlrich:DivClass}, $Y_c$ is locally a complete intersection and normal. By shrinking $A$ if necessary, we may assume $Y_c$ is irreducible. Note that
	$$\omega_{Y_c}\cong I_X\cdot\sO_{Y_c}=\sO_{Y_c}(-Z)$$
	where $Z=X\cap Y_c$. Hence $\sO_{Y_c}(-Z)$ is invertible on $Y_c$. Let $i:Y_c\rightarrow A$ be the inclusion morphism. Restricting the surjective morphism $v$  of Construction \ref{def:03} to $Y_c$ yields a surjective morphism 
	$$\oplus^r\sO_{Y_c}\stackrel{i^*f_1,\cdots, i^*f_r}{\superlongrightarrow}\sO_{Y_c}(-Z)\longrightarrow 0. $$
	Since $\sO_{Y_c}(-Z)$ is invertible, we see that the linear system $i^*|f_1,\cdots, f_r|$ can be decomposed as
	$$i^*|f_1,\cdots, f_r|=W_c+Z,$$
	where $W_c$ is a free linear system on $Y_c$. Now a general $t$-residual intersection $Y$ can be obtained as the locus of $t-c$ general sections of $W_c$ in $Y_c$.
	
	If $X$ is MJ-log canonical (resp, MJ-canonical), then so is $Y_c$ by the main result of \cite{Niu:MJLinkage}. As $Y_c$ is normal and locally a complete intersection, $Y_c$ has log canonical (resp. canonical) singularities \cite[Remark 2.5]{Ein:SingMJDiscrapency}. Now a Bertini type theorem can be established for a general divisor $H\in W_c$. Indeed, take a log resolution $f:\overline{Y}_c\longrightarrow Y_c$ of $Y_c$. The linear system $f^*W_c$ is base point free on $\overline{Y}_c$. So we take a general member $\overline{H}\in f^*W_c$, which is nonsingular. Accordingly, $\overline{H}$ gives a divisor $H\in W_c$. By the generality of $\overline{H}$, we see that $f^*(H)=\overline{H}$ as Cartier divisors. Now the adjunction formula gives us that 
	$$K_{\overline{H}/H}=K_{\overline{Y}_c/Y_c}|_{\overline{H}}.$$
	Thus $H$ is log canonical (resp. canonical) if so is $Y_c$. Finally a general $t$-residual intersection $Y$ is cut by $t-c$ general members of $W_c$. So by the above Bertini type theorem, we conclude that a general $t$-residual intersection $Y$ is locally a complete intersection with MJ-log canonical (resp. MJ-canonical) singularities, completing the proof.
\end{proof}

Finally, we propose the following conjecture predicting that MJ-singularities are preserved under a general residual intersection.
\begin{conjecture}\label{conj:01}
	Let $Y$ be a general $t$-residual intersection of $X$. Then $Y$ is MJ-canonical (resp, MJ-log canonical) if so is $X$.
\end{conjecture}

\bibliographystyle{amsalpha}

\end{document}